\documentclass[12pt,leqno]{amsart}

 \usepackage[utf8]{inputenc} 
\usepackage{graphicx}
\usepackage{indentfirst,csquotes}
\usepackage{float}

\usepackage{amssymb,amsthm,amsmath}
\usepackage{xcolor,paralist,hyperref,fancyhdr,etoolbox}
\newtheorem{theorem}{Theorem}[]
\newtheorem{definition}[theorem]{Definition}

\newtheorem{lemma}[theorem]{Lemma}
\newtheorem{proposition}[theorem]{Proposition}

\newtheorem{remark}[theorem]{Remark}

\hypersetup{ colorlinks=true, linkcolor=black, filecolor=black, urlcolor=black }

\def \C{\mathbb{C}}
\def \K{\mathbb{K}}
\def\id{\mathrm{Id}}
\def\su{\mathfrak{su}}
\def\sp{\mathfrak{sp}}
\def\spin{\mathfrak{spin}}
\def\f4{\mathfrak{f}_4}
\def\g{\mathfrak{g}}
\def\k{\mathfrak{k}}

\def\Aut{\mathrm{Aut}}
\def\R{\mathbb{R}}
\def\C{\mathbb{C}}
\def\H{\mathbb{H}}
\def\O{\mathbb{O}}
\def\Ric{\mathrm{Ric}}
\def\P{\mathrm{P}}

\def\be{\begin{equation}}
\def\ee{\end{equation}}

\begin{document}
\title{The Index of Cubic Focal Manifolds} 

\author{Niklas Rauchenberger, Uwe Semmelmann}

\address{Niklas Rauchenberger, Institut f\"ur Geometrie und Topologie, Fachbereich Mathematik, Universit{\"a}t Stuttgart, Pfaffenwaldring 57, 70569 Stuttgart, Germany}
\email{niklas.rauchenberger@igt.uni-stuttgart.de}

\address{Uwe Semmelmann, Institut f\"ur Geometrie und Topologie, Fachbereich Mathematik, Universit{\"a}t Stuttgart, Pfaffenwaldring 57, 70569 Stuttgart, Germany
}
\email{uwe.semmelmann@mathematik.uni-stuttgart.de}

\date{\today}

\let\thefootnote\relax
\footnotetext{MSC2020: Primary 53A10, Secondary 53C35.} 

\begin{abstract}
We calculate the index and nullity of the three orientable focal manifolds of isoparametric hypersurfaces in spheres with three distinct principal curvatures. It turns out that the index is equal to the dimension of the ambient Euclidean space and the nullity is  completely 
determined by the normal part of Killing vector fields of the ambient sphere. In that sense, the Veronese embeddings of the projective planes are as stable as possible for non totally geodesic submanifolds of the sphere.
\end{abstract}

\maketitle

\section{Introduction}
Minimal submanifolds are critical points of the volume functional. A natural question to ask is whether a given minimal submanifold realizes
a local minimum of the functional. This can be answered by studying the Jacobi operator, a second order differential operator that arises when computing the second variation of the volume functional. More precisely, one is interested in the index and nullity of the minimal submanifold, quantities that count the number of eigenvalues of this operator which are negative and zero, respectively. It is an interesting problem to study certain families of minimal submanifolds in ambient manifolds that carry nice geometric structures. As a classical result, Simons computed the index and nullity of (totally geodesic) spheres in higher-dimensional ambient spheres in his seminal work \cite{Simons1968}. Further examples include certain classes of totally geodesic submanifolds in compact symmetric spaces considered by Ohnita in \cite{Ohnita1987}, Kimura in \cite{Kimura2008}, as well as minimal submanifolds in Berger spheres (see Torralbo, Urbano \cite{Torralbo2022}).

Other examples of  minimal submanifolds to study are minimal isoparametric hypersurfaces in spheres. These are classical objects that possess a rich geometric structure.
As one of their most important properties, Münzner showed in \cite{Muenzner1980} that they can only have $g=1,2,3,4,6$ distinct principal curvatures, i.e., eigenvalues of the shape operator. In the cases $g=1,2$, corresponding to totally geodesic spheres and the Clifford torus, the index and the nullity have
already been calculated in the aforementioned work of Simons. If $g=3$, the situation becomes more involved. In \cite{Solomon1990}, Solomon calculated index and nullity in this case (see Table \ref{Table Solomon}). There is not much known about the index and nullity in the cases $g=4$ and $g=6$ apart from a result by Ball, Madnick, Semmelmann
where two specific examples are considered  for $g=4$ (see \cite{Ball2023}). Homogeneous (and thus all) isoparametric hypersurfaces in spheres with $g=3$
were already classified by Cartan in \cite{Cartan1939}. In the case $g=4$ there are also many inhomogeneous examples, most of which can be described by a construction of Ferus, Karcher, M\"unzner (see \cite{Ferus1981}). For $g=6$, the only inhomogeneous examples occur in the sphere $\mathbb{S}^{31}$ and this is also the last remaining open case in the classification of isoparametric hypersurfaces in spheres (see \cite{Siffert2016}, where a counterexample to Miyaoka's proof of the classification is given).

In this article, we  study focal submanifolds of isoparametric hypersurfaces in spheres. These are certain minimal submanifolds that are closely related to the isoparametric hypersurfaces (see Section 2.2 for details). In the cases of $g=1,2$ they are trivial -- that is, they reduce to single points or spheres -- but for $g=3$, which we will call the cubic case, these submanifolds have a rich structure: they are the classical Veronese embeddings of projective planes $\K \P^2$ over the four normed division algebras $\mathbb{K}=\mathbb{R},\mathbb{C},\mathbb{H},\mathbb{O}$ (we will exclude $\R \P^2$, however, as this space is not orientable). 

The main result of our article is the computation of  the index $  \mathrm{Ind}(\K \P^2) $, the nullity $\mathrm{Nul}(\K \P^2)$
 and the Killing nullity $\mathrm{Nul}_{\mathrm{K}}(\K \P^2)$, which arises from the normal parts of ambient Killing vector fields and gives a lower bound for the nullity.
 
\begin{theorem}
\label{Main Theorem}
For each orientable cubic focal manifold $\K \P^2\subset\mathbb{S}^{n}$, $n = \frac{3}{2}d+1$, of dimension $d=4,8,16$, we have
\begin{itemize}
\item[(a)]    $\mathrm{Ind}(\K \P^2) = n+1$.
\medskip
\item[(b)]     Furthermore, we have
\begin{itemize}
\item[(i)] $\mathrm{Nul}(\C \P^2)=\mathrm{Nul}_{\mathrm{K}}(\C \P^2)=20$,
\item[(ii)] $\mathrm{Nul}(\H \P^2)=\mathrm{Nul}_{\mathrm{K}}(\H \P^2)=70$,
\item[(iii)] $\mathrm{Nul}(\O \P^2)=\mathrm{Nul}_{\mathrm{K}}(\O \P^2)=273$.
\end{itemize}
\end{itemize}
\end{theorem}

\medskip
%

Simons proved in \cite{Simons1968}, Theorem 5.1.1, that the index of any compact minimal submanifold in the sphere is bounded
from below by its codimension, with equality  attained precisely for totally geodesic spheres. In the non-totally geodesic case, this
result was improved by El Soufi in  \cite{ElSoufi1993}. For a compact minimal hypersurface $M \subset \mathbb{S}^n$ he proved that
the lower bound of the index is $n+2$, with equality  attained on the Clifford torus (see  \cite{ElSoufi1993}, Theorem 2.2).
In the general case of a compact minimal submanifold in $\mathbb{S}^n$, El Soufi  proved the lower bound $n+1$ 
 (see  \cite{ElSoufi1993}, Theorem 2.1). He also poses the question whether equality can be attained. 
Our Theorem \ref{Main Theorem} answers this question and shows that the cubic focal manifolds  $\K \P^2$
 realize the equality case in the estimate, i.e., in a sense, they are as little unstable as possible.
 The $n+1$ destabilizing directions come from restricting the $n+1$ linearly independent  parallel vector fields of $\R^{n+1}$ to the sphere $\mathbb{S}^n$ and projecting them to the normal space of the submanifold (see  \cite{Simons1968}, Lemma 5.1.4). In the
 non-totally geodesic case, it is shown in \cite{ElSoufi1993} that these projections remain linearly independent.
 
 \medskip
 
 Finally, we also note the interesting fact that the nullities of the cubic focal manifolds given in 
 Theorem \ref{Main Theorem} agree with those of the corresponding minimal isoparametric hypersurfaces
(see Table \ref{Table Solomon} and \cite{Solomon1990}).

\medskip

The paper is structured as follows. In Section \ref{Preliminaries} we recall the notions of index and nullity of minimal submanifolds as well as some of the most important geometric properties of isoparametric hypersurfaces and their focal manifolds. Furthermore, we have a more detailed look in the case $g=3$. We also summarize basic results of harmonic analysis relevant to our application and present relations between the Laplacian and Casimir operators which will be our main technical tool for the computations. In Section \ref{Computation of the Index} we rewrite the Jacobi operator of the focal submanifolds as a sum of Casimir operators. We then use spectral computations from the literature to determine index and nullity.\\

\textbf{Acknowledgments. }
The second author would like to thank Gavin Ball and Jesse Madnick for initial discussions on focal manifolds, and in particular for the opportunity to pursue the cubic case independently. We are grateful to Marco Radeschi for valuable discussions related to the subject of this paper. We also thank Ivan Solonenko for a careful reading of the manuscript and for numerous suggestions that improved the presentation. Finally, we thank Jesse Madnick and
Jakob Hecker for helpful comments.

\section{Preliminaries}
\label{Preliminaries}

\subsection{Minimal Submanifolds}
Let $(\overline{M},\bar{g})$ be an $n$-dimensional Riemannian manifold and let $M$ be an isometrically immersed $d$-dimensional orientable compact submanifold of $\overline{M}$. Denote by ${\varphi:M\rightarrow\overline{M}}$ the isometric immersion and let $g:=\varphi^{*}\bar{g}$ be the induced metric on $M$. We will always write $\overline{\nabla}$, $\overline{R}$ for the Levi-Civita connection and curvature of $(\overline{M}, \bar{g})$ and $\nabla$, $R$ for those of $(M, g)$. Consider a smooth $1$-parameter variation ${\varphi_t}$ of the immersion $\varphi$ with $|t|<\varepsilon$ such that $\varphi_0=\varphi$. We call $(M,g)$ a \emph{minimal} submanifold of $(\overline{M},\overline{g})$ if the first derivative of the volume functional $\frac{d}{dt}\mathrm{vol}((M,\varphi_t^{*}\overline{g}))|_{t=0}$ vanishes. This is equivalent to  the vanishing of the mean curvature vector field of $M$. Naturally, one is then interested in the formula for the second variation for minimal submanifolds. Denote by $V :=\frac{d\varphi_t}{dt}|_{t=0}$ the variational vector field along $\varphi$, then one can show (see e.g. \cite{Simons1968}, Section 3.2)
\begin{equation*}
\frac{d^2}{dt^2}\mathrm{vol}((M,\varphi_t^{*}\overline{g}))|_{t=0}=\int_M g\big(-\Delta^{\perp}V^{\perp} + \mathrm{Ric}^{\perp}(V^{\perp})-\mathcal{A}(V^{\perp}),V^{\perp}\big)\,d\mathrm{vol}.
\end{equation*}
\label{Second Variational Formula}Here, $-\Delta^{\perp}=(\nabla^{\perp})^{*}\nabla^{\perp}$, where $^{*}$ denotes taking the $L^2$-adjoint, is the normal Laplacian defined from the normal connection $\nabla_X^{\perp}\xi:=(\overline{\nabla}_X\xi)^{\perp}$ for vector fields $X\in\Gamma(TM)$
and sections  $\xi\in\Gamma(T^\perp M)$ of the normal bundle of $M \subset \overline M$. Furthermore, for a local orthonormal frame $\{X_i\}$ of $TM$, the second summand is a Ricci-like curvature term, defined as  $\mathrm{Ric}^{\perp}(\xi):=\sum_{i=1}^d(\overline{R}(X_i,\xi)X_i)^{\perp}$.
 The third summand $\mathcal{A}(\xi):=\sum_{i=1}^d\mathrm{II}(A_{\xi}X_i,X_i)$ involves the second fundamental form $\mathrm{II}$ of $M$ in $\overline{M}$ as well as the shape operator $A_{\xi}X:=-(\overline{\nabla}_X\xi)^{\top}$ with a normal vector field $\xi$ and a tangent vector   $X$. We introduce the \emph{Jacobi operator} $\mathcal{J}$ acting on sections of the normal bundle $T^\perp M$ as
\begin{equation*}
\mathcal{J}:=-\Delta^{\perp}+\mathrm{Ric}^{\perp}-\mathcal{A}.
\end{equation*}
\label{Jacobi Operator}The Jacobi operator is an elliptic self-adjoint operator. Hence, on a compact manifold, it has a discrete spectrum with eigenvalues of finite multiplicity. The {\it index} of $M$ is defined as the number of negative eigenvalues (counted with multiplicity) and  denoted by $\mathrm{Ind}(M)$, which corresponds to directions in the normal bundle where the volume of $M$ decreases. Moreover, the \emph{nullity} of $M$ is defined as the multiplicity of the eigenvalue $0$ and  denoted by $\mathrm{Nul}(M)$. Killing fields on $\overline{M}$ (restricted to $M$ and projected onto the normal bundle) always give volume preserving variations and thus contribute to the nullity of $M$. This motivates the definition of the \emph{Killing nullity} of $M$, denoted by $\mathrm{Nul}(M)_{\mathrm{K}}$, as
the dimension of the space of normal components of  Killing fields of  $(\overline{M}, \bar g)$.

\subsection{Geometry of Isoparametric Hypersurfaces and their Focal Manifolds}\label{focal1}
In this section, we will recall some fundamental properties of isoparametric hypersurfaces in spheres and their focal manifolds. We discuss their geometry in general before we specialize to the cubic case in the next section. The facts  presented here  and more details can be found for example in the books \cite{Berndt2016}, Section 2.9 and 3.4.2, \cite{Cecil2015}, Section 3 or in the articles \cite{Li2015}, \cite{Li2016}, \cite{Tang2013}.
\begin{definition}
An \emph{isoparametric function on} $\mathbb{S}^n$ is a nonconstant smooth function $f:\mathbb{S}^n\rightarrow\mathbb{R}$ such that $|\nabla f|^2=a\circ f$ for some smooth function $a:I\rightarrow\mathbb{R}$ and $\Delta f=b\circ f$ for some continuous function $b:I\rightarrow\mathbb{R}$, where $I=f(\mathbb{S}^n)\subset\mathbb{R}$ is an interval.
\end{definition}
One can show that each point $c$ in the interior of $I$ is a regular value of $f$. The regular level set $M_c:=f^{-1}(c)=\{p\in\mathbb{S}^n\,|\,f(p)=c\}$ is called an \emph{isoparametric hypersurface} of $\mathbb{S}^n$. The collection $\{f^{-1}(c)\}_{c\in I}$ is called an \emph{isoparametric family}. One can define isoparametric hypersurfaces in Euclidean or hyperbolic space in an analogous fashion. In these spaces isoparametric hypersurfaces are well understood and fully classified. However, in the spherical case, the study of isoparametric hypersurfaces has been much more convoluted. Cartan started this investigation but he was only able to classify them in the cubic case and even today the classification is not entirely completed (see \cite{Chi2020} for more details on the history).

Isoparametric hypersurfaces in spheres -- like in all space forms -- have constant principal curvatures, i.e., eigenvalues of the shape operator, by a theorem of Cartan (see \cite{Cartan1938}). M\"unzner showed that the number $g$ of distinct principal curvatures can only be $g=1,2,3,4,6$. We can write them as $\cot\theta_i$ for some $\theta_1,\ldots,\theta_g$. Then, by a result of Münzner (see \cite{Muenzner1980}), we have
\begin{equation*}
\theta_i=\theta_1+\frac{(i-1)}{g}\pi,\quad 1\leq i\leq g.
\end{equation*}
The multiplicities of these eigenvalues satisfy $m_i=m_{i+2}$ (with the index $i\,\mathrm{mod}\,g$). Now, let $\xi$ be a unit normal vector field and consider the map
\begin{equation*}
\varphi_t:M_c\rightarrow\mathbb{S}^n,  x \mapsto \,(\cos t)x+(\sin t)\xi(x).
\end{equation*}
If $\cot t$ is not a princial curvature of $M_c$, then $\varphi_t(M_c)$ is also an isoparametric hypersurface of $\mathbb{S}^n$, called \emph{parallel} to $M_c$, with principal curvatures $\cot(\theta_1-t),\ldots,\cot(\theta_g-t)$. If however, $\cot t=\cot \theta_i$ for some $i=1,\ldots , g$, then $\varphi_t(M_c)$ is a submanifold of higher codimension $m_i+1$, called a \emph{focal submanifold} of $\mathbb{S}^n$, with principal curvatures $\cot(\frac{k\pi}{g})$ for $1\leq k\leq g-1$. One of the important properties of focal submanifolds is that they have isospectral shape operators and are austere (i.e., for every eigenvalue $\lambda$ for any of its shape operators, $-\lambda$ is also an eigenvalue), thus they are  minimal (\cite{Muenzner1980}).

As it turns out, the isoparametric functions $f$ on the sphere $\mathbb{S}^n$ are given by restrictions of homogeneous polynomials on the ambient $\mathbb{R}^{n+1}$ of degree $g$, that satisfy certain differential equations (see \cite{Muenzner1980}). They are called the Cartan-Münzner polynomials. As one of their properties, their range is always $[-1,1]$. Furthermore, in each isoparametric family there are always exactly the two focal submanifolds $M_{-}:=f^{-1}(-1)$ and $M_{+}:=f^{-1}(1)$ corresponding to the singular values of the isoparametric function ${f:\mathbb{S}^n\rightarrow\mathbb{R}}$. It follows easily from the explicit form of their principal curvatures, that they are minimal submanifolds and furthermore that there is exactly one isoparametric hypersurface $f^{-1}(c_0)$ for some ${c_0\in(-1,1)}$ in the isoparametric family, that is  minimal too.

In this article we will  be concerned with those isoparametric families that are homogeneous, i.e., that are the orbits of some isometric cohomogeneity-one action on the ambient sphere. A result by Hsiang and Lawson (see \cite{Hsiang1971}, Theorem 5) shows that a isoparametric hypersurface in $\mathbb{S}^n$ is homogeneous if and only if it is a principal orbit of the isotropy representation of a Riemannian symmetric space of rank two.


\subsection{Cubic Focal Manifolds}

From now on we want to specialize to the situation of hypersurfaces with $g=3$ distinct principal curvatures and their focal manifolds. In this case, Cartan showed that there are only four minimal isoparametric hypersurfaces:
\bgroup
\def\arraystretch{1.5}
\begin{table}[h]
\begin{tabular}{|c|c|c|}
\hline
\textbf{Minimal Isoparametric Hypersurface}                                              & \textbf{Index} & \textbf{Nullity} \\ \hline \hline $\frac{\mathrm{SO}(3)}{\mathbb{Z}_2\times\mathbb{Z}_2}\subset\mathbb{S}^4$  & $20$             & $7$                \\ \hline
$\,\frac{\mathrm{SU}(3)}{T^2}\,\subset\mathbb{S}^7$                              & $44$             & $20$               \\ \hline
$\,\,\,\frac{\mathrm{Sp}(3)}{\mathrm{Sp}(1)^3}\subset\mathbb{S}^{13}$               & $119$            & $70$               \\ \hline
$\,\,\frac{\mathrm{F}_4}{\mathrm{Spin}(8)}\subset\mathbb{S}^{25}$                                                           & $377$            & $273  $            \\ \hline
\end{tabular}\\[1ex]
\caption{Minimal isoparametric hypersurfaces of the sphere with $g=3$ distinct principal curvatures and their index and nullity (see \cite{Solomon1990}).}
\label{Table Solomon}
\end{table}
\egroup

Furthermore, Cartan showed that the focal manifolds are diffeomorphic to $\mathbb{K}\P^2$ via the Veronese embeddings (see Table \ref{Table Cubic Focal Manifolds}). For this reason, we will always denote these focal manifolds as $\K \P^2$ with $\mathbb{K}\in\{\mathbb{R},\mathbb{C},\mathbb{H},\mathbb{O}\}$.

\bgroup
\def\arraystretch{1.5}
\begin{table}[H]
\begin{tabular}{|c|c|}
\hline
\textbf{Focal Manifold} & \textbf{Minimal Isoparametric Hypersurface} \\ \hline \hline
$\R \P^2=\frac{\mathrm{SO}(3)}{\mathrm{O}(2)}\subset\mathbb{S}^4$                                    & $\frac{\mathrm{SO}(3)}{\mathbb{Z}_2\times\mathbb{Z}_2}$         \\ \hline
$\C \P^2=\frac{\mathrm{SU}(3)}{\mathrm{U}(2)}\subset\mathbb{S}^7$                                    & $\frac{\mathrm{SU}(3)}{T^2}$                                     \\ \hline
\,\,\,\,\,\,\,\,\,\,\,\,\,$\H \P^2=\frac{\mathrm{Sp}(3)}{\mathrm{Sp}(2)\cdot\mathrm{Sp}(1)}\subset\mathbb{S}^{13}$                & $\frac{\mathrm{Sp}(3)}{\mathrm{Sp}(1)^3}$                       \\ \hline
\,\,\,\,\,$\O \P^2=\frac{\mathrm{F}_4}{\mathrm{Spin}(9)}\subset\mathbb{S}^{25}$                                  & $\frac{\mathrm{F}_4}{\mathrm{Spin}(8)}$                 \\ \hline
\end{tabular}\\[1ex]
\caption{Focal manifolds of minimal isoparametric hypersurfaces of spheres with $g=3$ distinct principal curvatures.}
\label{Table Cubic Focal Manifolds}
\end{table}
\egroup

\vspace{-.8cm}

\begin{remark}
In the following we will exclude $\R \P^2$ since it is non-orientable. One can also define index and nullity of non-orientable submanifolds, but then the  volume form has to be replaced by the Riemannian density. We refer to \cite{Montiel1997}, where  the case of $\R \P^2$ has already been (implicitly) dealt with.
\end{remark}

\medskip

The normal bundle of the $d$-dimensional focal manifold $\K \P^2 = G/K$ has rank $\frac{d}{2} + 1$. It is isomorphic to a homogeneous vector bundle $G\times_{\pi}W$, where the (irreducible) slice representations $\pi:K\rightarrow \Aut (W)$ can be read off from Tables E and F in \cite{Grove2008}:
\begin{equation}
\label{slice representations}
\pi:\mathrm{SU}(2)\rightarrow\mathrm{Aut}(\mathfrak{su}(2)),\quad\pi:\mathrm{Sp}(2)\rightarrow\mathrm{Aut}((\Lambda^2_0(\mathbb{C}^4))_{\mathbb{R}}),\quad\pi:\mathrm{Spin}(9)\rightarrow\mathrm{Aut}(\mathbb{R}^9).
\end{equation}

\medskip

For  $g=3$  the shape operators $A_\xi$  of the focal manifolds  $\K \P^2$ all have eigenvalues 
$\pm\frac{1}{\sqrt{3}}$, where  $\xi$ is a normal vector field of length one. 
Moreover, one can show (see \cite{Console1998}, Section 2) that for an orthonormal basis $\xi_0, \ldots ,\xi_{\frac{d}{2}}$
of $T^\perp_p \K \P^2$ with $p\in \K \P^2$, the symmetric endomorphisms $P_0:=\sqrt{3}A_{\xi_0},\ldots,P_{\frac{d}{2}}:=\sqrt{3}A_{\xi_{\frac{d}{2}}}$ 
of a tangent space $T_p \K \P^2$  form a Clifford system, i.e. ,
\be\label{cliff}
P_i^2= \id \qquad \mbox{and} \qquad   P_iP_j+P_jP_i=0   \quad \mbox{for} \,\;  i \neq j.
\ee

%
%


\medskip

On each focal manifold $\K \P^2 = G/K$ we have two $G$-invariant metrics: the metric defined by the negative 
Killing form of $\g$, the Lie algebra of $G$, and the metric $g_\K$ induced from the round metric on the ambient sphere.
Both metrics are normal and, since the symmetric spaces $G/K$ are irreducible, they differ by a constant.
For our computations below we need the explicit value of this constant. Using \eqref{cliff}, we can deduce this information easily as the next lemma shows.

\begin{lemma}
\label{Focal Metrics and Normal Bundles}
For each orientable focal manifold $\K \P^2=G/K$, let $b_{\mathbb{K}}$ be the inner product on $\mathfrak{g}$ that induces the metric $g_{\mathbb{K}}$, and let $b_\g$ be the negative of the Killing form of $\mathfrak{g}$. Then,
\begin{equation*}
b_{\mathbb{C}}  =  \frac{1}{4}  \, b_{\su(3)},  \quad b_{\mathbb{H}}  =  \frac{3}{32} \,  b_{\sp(3)}, \quad 
b_{\mathbb{O}}  =  \frac{1}{24}\,  b_{\f4} .
\end{equation*}
\end{lemma}
\begin{proof}
Let $\{X_i\}_{i =1}^d$  be a local orthonormal frame of $T\K \P^2$. Summing the Gau\ss{}  equation over pairs
of vector fields $X_i, X_j, i\neq j$, the explicit form of the curvature tensor on the ambient sphere and minimality of the focal manifold gives
\begin{equation*}
d(d-1)  =  \mathrm{scal}_{g_\mathbb{K}}  + \sum_{i,j=1}^d\Big(|\mathrm{II}(X_i,X_j)|^2-\left\langle\mathrm{II}(X_j,X_j),\mathrm{II}(X_i,X_i)\right\rangle\Big)\!=\mathrm{scal}_{g_\mathbb{K}}  +  \sum_{i,j=1}^d|\mathrm{II}(X_i,X_j)|^2,
\end{equation*}
Plugging in  $\mathrm{II}(X,Y)=\sum_{i=0}^{d/2}\left\langle A_{\xi_i}X,Y\right\rangle\xi_i$, where $X,Y\in\Gamma(T\K \P^2)$ and $\{\xi\}_{i=0}^{d/2}$ is a local orthonormal frame of the normal bundle $T^\perp \K \P^2$, leads to
\begin{equation*}
\mathrm{scal}_{g_\mathbb{K}}
= d(d-1)  -   \sum_{i,j=1}^d\sum_{k=0}^{d/2}\left\langle A_{\xi_k}X_i,X_j\right\rangle^2
=d (d-1)  - \frac{d}{3}\left(\frac{d}{2}+1\right),
\end{equation*}
\noindent
where we used the  Clifford relations  \eqref{cliff} to compute the sum.
Hence, we obtain  $\mathrm{scal}_{g_\mathbb{C}}=8$, $\mathrm{scal}_{g_\mathbb{H}}=\frac{128}{3}$ and $\mathrm{scal}_{g_\mathbb{O}}=192$. Since on a symmetric space $G/K$ the scalar curvature of the metric defined by the negative Killing form on $\mathfrak{g}$ is given by $\frac{\mathrm{dim}(\K \P^2)}{2}$, we can now easily compute the constants given in the lemma.
\end{proof}


\begin{remark}
There are several nice geometric properties of the focal metrics $g_\K$. Firstly, their sectional curvatures are bounded below by $\frac{1}{3}$ and above by $\frac{4}{3}$ (see \cite{Li2016a}, Remark 2.2). 
\\Secondly, with the above normalizations it holds that $\lambda_1(\K \P^2)=\mathrm{dim}(\K \P^2)$, where $\lambda_1(\K \P^2)$ is the first non-zero eigenvalue of the Laplacian on functions 
(see \cite{Milhorat2023}, p. 2).
\end{remark}

\subsection{Representation Theory}
In this section we will briefly recall basics of harmonic analysis on homogeneous spaces and introduce the Casimir operator. As we will see, this simplifies the analytic problem of computing the spectrum of the Jacobi operator to a representation-theoretic task. 

Let $M=G/K$ be a homogeneous space where $G$ is a compact connected Lie group and $K$ a closed subgroup of $G$. Given a unitary representation $\pi:K\rightarrow\mathrm{Aut}(V)$, we can associate to the natural principal $K$-bundle $G\rightarrow M,\,g\mapsto gK$ the homogeneous vector bundle $VM:=G\times_{\pi} V$. On the space of $L^2$-sections $L^2(M,VM)$, the group $G$ acts via pullback, while on the space of $K$-equivariant $L^2$-functions
\begin{equation*}
L^2(G,V)_K:=\{f\in L^2(G,V)\:|\:f(gk)=\pi(k)^{-1}f(g),g\in G,k\in K\},
\end{equation*}
the group $G$ acts by left-translation, i.e., $(g\cdot f)(h)=f(g^{-1}h)$ for $g,h\in G$, $f\in L^2(G,V)_K$. These representations of $G$ can be canonically identified under the isomorphism
\begin{equation*}
L^2(M,VM)\rightarrow L^2(G,V)_K,\,\xi\mapsto\widehat{\xi}
\end{equation*}
given by $\xi(gK)=[g,\widehat{\xi}(g)]$ for $g\in G$.

Let $D(G)$ denote the set of equivalence classes of finite-dimensional irreducible complex representations of $G$ and denote the representatives by $\lambda:G\rightarrow\mathrm{Aut}(V_{\lambda})$. As a consequence of the Peter-Weyl theorem and Frobenius reciprocity we get
\begin{equation}
\label{Peter-Weyl}
L^2(M,VM)\cong\overline{\bigoplus_{\lambda\in D(G)}}V_{\lambda}\otimes\mathrm{Hom}_K(V_{\lambda},V),
\end{equation}
where $\mathrm{Hom}_K(V_{\lambda},V)$ denotes the space of $K$-equivariant maps, i.e., the space of all linear maps $L:V_{\lambda}\rightarrow V$ such that $\pi(k)\cdot L=L\cdot\lambda(k)$ for $k\in K$. 
%

\begin{definition}
\label{Definition Casimir Operators}
Let $b$ be an $\mathrm{Ad}(G)$-invariant inner product on $\g$ and let
$\{X_i\}_{i=1}^m$ be a $b$-orthonormal basis of $\mathfrak{g}$. For a representation $\pi:G\rightarrow\mathrm{Aut}(V)$ the endomorphism
\begin{equation*}
\mathrm{Cas}_{V}^{G,b}:=-\sum_{i}\big(d\pi(X_i)\big)^2
\end{equation*}
is called the \emph{Casimir operator of $\pi$ with respect to $b$}.
\end{definition}

Note that our choice of sign in Definition \ref{Definition Casimir Operators} makes $\mathrm{Cas}_{V}^{G,b}$ a positive definite operator. The Casimir operator is a $G$-equivariant endomorphism. Hence, 
by Schur's lemma, it acts on irreducible representations as a constant times the identity. 
The corresponding constants can be calculated by means of the root theory. Choose a maximal torus $\mathfrak{t}$ of $\mathfrak{g}$, make a choice of positive roots $\Sigma^{+}\subset\Sigma\subset(i\mathfrak{t})^{*}$ and define the half sum of positive roots 
$\rho:=\frac{1}{2}\sum_{\alpha\in\Sigma^{+}}\alpha$. The inner product $b$ on $\mathfrak{g}$ induces a (positive definite) inner product on $(i\mathfrak{t})^{*}$, which we will also denote by $b$. Let $V_\lambda$ be an irreducible representation of highest weight $\lambda \in (i\mathfrak{t})^{*}$.
Then, by the well-known \emph{Freudenthal formula},
\be\label{freudenthal}
\mathrm{Cas}_{V_{\lambda}}^{G,b}
=c_{\lambda}\,\mathrm{Id}_{V_{\lambda}} \:\:\text{ with }\:\: c_{\lambda}:=b(\lambda,\lambda+2\rho) .
\ee
Note that $\mathrm{Cas}_{V}^{G,c\, b} = \frac{1}{c} \, \mathrm{Cas}_{V}^{G,b}$. This can be seen from \eqref{freudenthal} or directly from Definition
\ref{Definition Casimir Operators}.

%
%
%
%
%

The Casimir operator of the adjoint representation (with respect to minus the Killing form) acts as identity.
This can of course be shown without using \eqref{freudenthal}.
Moreover, we have the so-called \textit{strange formula} of  Freudenthal and de Vries  (\cite{Freudenthal1969}, 47.11)
\begin{equation}
\label{Strange Formula}
b_\g(\rho,\rho)=\frac{\mathrm{dim}(\mathfrak{g})}{24}.
\end{equation}
Recall that $b_\g$ is our notation for the negative Killing form of $\g$.

%
%
%
%

\medskip

Let $M=G/K$ be a normal homogeneous space (still with $G$ compact), that is, $M$ is a homogeneous space with reductive decomposition $\mathfrak{g}=\mathfrak{k}\oplus\mathfrak{m}$ such that the metric $g$ on $M$ is determined by an $\mathrm{Ad}(G)$-invariant inner product $b$ on $\mathfrak{g}$. Denote by $\overline{\nabla}$ the canonical connection of $M$ (which in general differs from the Levi-Civita connection of $M$ by a torsion term). We introduce the \emph{standard Laplacian} $\bar{\Delta}$ (see \cite{Semmelmann2019}) acting on $\Gamma(VM)$ by
\be\label{standard}
\bar{\Delta}=\bar{\nabla}^{*}\bar{\nabla}+q(\bar{R}),
\ee
where $q(\bar{R}) = \sum_{i<j} d\pi (X_i \wedge X_j) \circ d\pi (\bar{R}(X_i, X_j))
$,
for an orthonormal basis $\{X_i\}$, is the curvature endomorphism on $VM$.
The next proposition relates the Laplace operator $\bar \Delta$ to the Casimir operator (see \cite{MS10}, Lemma 5.2).

\begin{proposition}
\label{Laplace is Casimir}
Let $M=G/K$ be a normal homogeneous space with a compact Lie group $G$, that has Lie algebra $\mathfrak{g}$ with an $\mathrm{Ad}(G)$-invariant inner product $b$. Then,
\begin{equation*}
\bar{\Delta}=\mathrm{Cas}_{\Gamma(VM)}^{G,b}\quad\text{and}\quad q(\bar{R})=\mathrm{Cas}_{V}^{K,b}=\mathrm{Cas}_{\Gamma(VM)}^{K,b}.
\end{equation*}
\end{proposition}

To be more precise, we take the restricted inner product $b|_{\k}$ for the Casimir operator of the subgroup $K$. However, we will suppress this notation for better readability. Also, note that (\ref{Peter-Weyl}) does not hold for real representations $\pi:K\rightarrow\mathrm{Aut}(V)$, but instead we have to take the complexification $V^{\mathbb{C}}$ to apply the theory. This is, however, not a problem since all eigenvalues of $\bar{\Delta}=\mathrm{Cas}_{\Gamma(VM)}^{G,b}$ are real, so we can just take the real part of complex eigensections of $\bar{\Delta}$.

\medskip

\section{Computation of Index and Nullity}
\label{Computation of the Index}

In this section we come to the computation of index and nullity of all orientable cubic focal manifolds. The Jacobi operator has a particular nice form in our case and simplifies to a shifted Casimir operator. We can then use standard representation theoretic tools to calculate its eigenvalues and prove Theorem \ref{Main Theorem}.

\subsection{The Jacobi Operator}
\label{The Jacobi Operator}

We begin with the following two short calculations of the curvature and the second fundamental form terms in the Jacobi operator. Take ${\xi\in\Gamma(T^\perp\K \P^2)}$ and let $\{X_i\}$ be a local orthonormal frame of $T\K \P^2$. Since the curvature of the ambient sphere $\mathbb{S}^n$ is simply given by the formula $\overline{R}(X,Y)Z=-(X\wedge Y)Z$ for tangent vectors $X,Y,Z$, we can compute
\begin{equation*}
\mathrm{Ric}^{\perp}(\xi)=\sum_{i=1}^d(\overline{R}(X_i,\xi)X_i)^{\perp}
=-\sum_{i=1}^d\big((X_i\wedge\xi)X_i\big)^{\perp} =  -\sum_{i=1}^d |X_i|^2\xi =-d   \xi.
\end{equation*}
From the fact that the (scaled) shape operators form a Clifford system (see \eqref{cliff}), we get
\begin{equation*}
\mathcal{A}(\xi)=\sum_{i=1}^d\mathrm{II}(A_{\xi}X_i,X_i)=\sum_{i=1}^d\sum_{k=0}^{d/2}\left\langle A_{\xi_k}A_{\xi}X_i,X_i\right\rangle\xi_k=\frac{d}{3}\xi.
\end{equation*}
To treat the normal Laplacian, we 
use \eqref{standard} and Proposition \ref{Laplace is Casimir} to  rewrite 
$$
-\Delta^{\perp} = (\nabla^{\perp})^{*}\nabla^{\perp}=\mathrm{Cas}_{\Gamma(T^\perp \K \P^2)}^{G,b_{\mathbb{K}}}-\mathrm{Cas}_{\Gamma(T^\perp \K \P^2)}^{K,b_{\mathbb{K}}},
$$ where we take the Casimir operators with respect to the scalar product $b_{\mathbb{K}}$ on $\mathfrak{g}$ as determined in 
Lemma \ref{Focal Metrics and Normal Bundles}. In particular, the Casimir operator of $K$ is also taken with respect to $b_\K$, but restricted to the Lie algebra $\k$. Note that we can apply Proposition~\ref{Laplace is Casimir} since on symmetric spaces the  Levi-Civita connection and the canonical connection coincide. To summarize, we have obtained the following form of the Jacobi operator:
\begin{equation}
\label{Expanded Jacobi Operator}
\mathcal{J}=\mathrm{Cas}_{\Gamma(T^\perp\K \P^2)}^{G,b_{\mathbb{K}}}-\mathrm{Cas}_{\Gamma(T^\perp \K \P^2)}^{K,b_{\mathbb{K}}}-\frac{4}{3}d\,\mathrm{Id}\quad\text{on}\:\,\Gamma(T^\perp \K \P^2).
\end{equation}
From now on, the task of computing the spectrum of $\mathcal{J}$ is a purely algebraic one. 
However, we still have the following simplified expression for the Jacobi operator.

\begin{proposition}
\label{Simplified Jacobi Operator}
For each orientable cubic focal manifold $\K \P^2=G/K $ of dimension $d=4,8,16$, the Jacobi operator on $\Gamma(T^\perp \K \P^2)$ is given by
\begin{equation*}
\mathcal{J}=\mathrm{Cas}_{\Gamma(T^\perp\K \P^2)}^{G,b_{\mathbb{K}}}-2d \, \mathrm{Id} .
\end{equation*}
\end{proposition}
\begin{proof}
We have to calculate the Casimir operator of $K$ with respect to $b_{\mathbb{K}}|_\k$. From  Proposition \ref{Laplace is Casimir}
we know that $\mathrm{Cas}_{\Gamma(T^\perp \K \P^2)}^{K,b_\K} = \mathrm{Cas}_{W}^{K,b_\K}$, where $W$ is the (irreducible) slice representation 
on the normal bundle. Hence, the Casimir acts as a multiple of the identity on $\Gamma(T^\perp\K \P^2)$. The constant can be determined by using the Freudenthal formula.

\medskip

\noindent
\textbf{Case $\C \P^2$:} The slice representation of the complex projective plane $\C \P^2=\frac{\mathrm{SU}(3)}{\mathrm{U}(2)}$ is  the adjoint representation ${\pi:\mathrm{SU}(2)\rightarrow\mathrm{Aut}(\mathfrak{su}(2))}$. 
We already know that the Casimir operator with respect to $b_{\su(2)}$  acts as identity and from Lemma  
\ref{Focal Metrics and Normal Bundles}  we have $b_\C = \frac14 b_{\su(3)}$. It remains to compute the restriction of the Killing form of $\su(3)$ to $\su(2)$.
This can be done using \eqref{Strange Formula} or by noting that the Killing form of $\su(n)$ is $2n$ times the trace form. Hence,
$b_\C = \frac14 b_{\su(3)} =  \frac38 b_{\su(2)}$ and  we obtain
$
\mathrm{Cas}_{\mathfrak{su}(2)}^{\mathrm{SU}(2),b_{\mathbb{C}}}
=\frac{8}{3}\mathrm{Cas}_{\mathfrak{su}(2)}^{\mathrm{SU}(2),b_{\su(2)}}
=\frac{8}{3}\mathrm{Id}
=\frac{2}{3}d\,\mathrm{Id} .
$

\medskip

\noindent
\textbf{Case $\H \P^2$:} The slice representation of the quaternionic projective plane $\H \P^2=\frac{\mathrm{Sp}(3)}{\mathrm{Sp}(2)\cdot\mathrm{Sp}(1)}$ is ${\pi:\mathrm{Sp}(2)\rightarrow\mathrm{Aut}(\Lambda^2_0(\mathbb{C}^4))_{\mathbb{R}}}$ 
(the $\mathrm{Sp}(1)$ part acts trivially). 
We want to calculate the Casimir operator with respect to $b_\H|_{\sp(2)}$ using the Freudenthal formula \eqref{freudenthal}.
Then $b_\H|_{\sp(2)} = \frac{3}{32} b_{\sp(3)}|_{\sp(2)} = \frac18  b_{\sp(2)}$ by Lemma \ref{Focal Metrics and Normal Bundles}  
and the fact that the Killing form of $\sp(n)$ is $2(n+1)$ times the trace form. The inner product induced by $b_{\sp(2)}$ on
$(i \mathfrak{t})^* \cong \R^2$ differs from the standard scalar product $\left\langle\cdot,\cdot\right\rangle$ on $\R^2$ by a factor $c \in \R$. This factor can be determined from the Freudenthal de Vries formula \eqref{Strange Formula}. 
Indeed, the highest weight of the representation $\pi$ in terms of fundamental weights of $\mathfrak{sp}(2)$ is $\pi=\omega_2=\left(\begin{smallmatrix} 1\\ 1\\ \end{smallmatrix}\right)$ and furthermore, $\rho=\left(\begin{smallmatrix} 2\\ 1\\ \end{smallmatrix}\right)$, i.e., $\langle \rho, \rho \rangle = 5 $. Thus, 
$
b_{\sp(2)} (\rho, \rho) = 5 c = \frac{\mathrm{dim}(\mathfrak{sp}(2))}{24}=\frac{5}{12}
$,
i.e. $c = \frac{1}{12}$. 
The details on root systems and fundamental weights
for simple Lie algebras are taken from  \cite{Knapp2002}, Appendix C.
Substituting all this into the Freudenthal formula \eqref{freudenthal}, we obtain
\begin{equation*}
\mathrm{Cas}_{(\Lambda_0^2(\mathbb{C}^4))_{\mathbb{R}}}^{\mathrm{Sp}(2),b_\H}=\frac{8}{12}\left\langle\pi,\pi+2\rho\right\rangle=\frac{2}{3}\left\langle\begin{pmatrix}
1\\1\\ \end{pmatrix},\begin{pmatrix}
4\\2\\ \end{pmatrix}+\begin{pmatrix}
1\\1\\ \end{pmatrix}\right\rangle\mathrm{Id}=\frac{16}{3}\mathrm{Id} = \frac{2}{3}d\,\mathrm{Id}.
\end{equation*}

\medskip

\noindent
\textbf{Case $\O \P^2$:} The slice representation of the Cayley plane $\O \P^2=\frac{\mathrm{F}_4}{\mathrm{Spin}(9)}$ is the standard representation ${\pi:\mathrm{Spin}(9)\rightarrow\mathrm{Aut}(\mathbb{R}^9)}$. 
The highest weight of this representation in terms of fundamental weights of $\mathfrak{spin}(9)$ is 
$\pi=\omega_1= (\frac12, \frac12, \frac12, -\frac12)^T$ and $\rho = (4,2,1,0)^T$, i.e. $\langle \rho, \rho \rangle = 21$. 
The inner product on $(i \mathfrak{t}_{\spin(9)})^* \cong \R^4$ induced by $b_{\spin(9)}$ differs from the standard scalar product $\left\langle\cdot,\cdot\right\rangle$ on $\R^4$ by a constant $c \in \R$ . Using \eqref{Strange Formula}, we find ${21 c = \frac{\mathrm{dim}(\mathfrak{spin}(9))}{24}=\frac{3}{2}}$, which gives  $c=\frac{1}{14}$. 
%
%
%

We also have to compare  the restriction of $b_\O = \frac{1}{24} b_{\f4}$ to $\spin(9)$ with $b_{\spin(9)}$. This time we have to use  \eqref{Strange Formula}. 
We can choose the maximal torus of  $\mathrm{F}_4$ to coincide with $\mathfrak{t}_{\spin(9)}$. Then the half-sum $\rho$ of the positive roots of $\f4$ is given by 
$
\rho = (\frac{11}{2}, \frac{5}{2}, \frac{3}{2}, \frac{1}{2})^T
$
with $\langle \rho, \rho \rangle = 39$ and \eqref{Strange Formula}
shows that the inner product induced by $b_{\f4}$ on $(i \mathfrak{t}_{\spin(9)})^* \cong \R^4$ is $\frac{1}{18}$ times the
standard scalar product of $\R^4$. It follows that $b_{\f4}|_{\spin(9)} = \frac{9}{7} b_{\spin(9)}$ and thus
$b_{\O}|_{\spin(9)} = \frac{3}{56} b_{\spin(9)}$. Note that we have to take the reciprocal, since we compare the induced inner products on the dual space $(i \mathfrak{t}_{\spin(9)})^*$. Substituting these scalings into the Freudenthal formula  \eqref{freudenthal}  gives
\begin{equation*}
\mathrm{Cas}^{\mathrm{Spin}(9),b_\O}_{\mathbb{R}^9}
=\frac{14}{3} \left\langle\pi,\pi+2\rho\right\rangle=\frac{14}{3}\left\langle\begin{pmatrix}
\;\; 1/2\\ \;\; 1/2\\ \;\;1/2\\-1/2\\ \end{pmatrix},\begin{pmatrix}
8\\4\\2\\0\\ \end{pmatrix}+\begin{pmatrix}
\;\;1/2\\ \;\; 1/2\\ \;\; 1/2\\-1/2\\ \end{pmatrix}\right\rangle\mathrm{Id}
=\frac{32}{3}\mathrm{Id} 
=  \frac{2}{3}d\,\mathrm{Id} .
\end{equation*}
%
%
%
In each of the three cases the formula for $\mathcal{J}$ then follows from (\ref{Expanded Jacobi Operator}).
\end{proof}

\medskip

\begin{remark}
The above proposition can also be proven in a more geometric way: by Proposition \ref{Laplace is Casimir}
the Casimir operator of $K$ coincides with the curvature operator $q(R)$ of the three normal bundles. In the
case of $\C \P^2$ the normal bundle is a subbundle of the bundle of $2$-forms. For $\H \P^2$ and $\O \P^2$
the normal bundles are subbundles of the bundle of symmetric $2$-tensors. On both tensor bundles we have $q(R) = 2 \mathcal R + \Ric $, where $\Ric$ acts as derivation and $ \mathcal R$ is
a curvature operator whose eigenvalues  can be computed easily.
%
%
%
\end{remark}

\subsection{Proof of Theorem \ref{Main Theorem}}
\label{Proof of Theorem}

As a consequence of Proposition \ref{Simplified Jacobi Operator}, the remaining step is to find
all the representations $\lambda\in D(G)$ that have Casimir eigenvalue less than or equal to $2d$ and to compute their multiplicities. Up to the right scaling, we can take these values from the literature. This yields our main result.

\begin{proof}[Proof of Theorem \ref{Main Theorem}]
By (\ref{Peter-Weyl}) and Proposition \ref{Simplified Jacobi Operator}, we have the formula
\begin{equation*}
\mathrm{Ind}(M)=\sum_{\lambda\in D(G),\,c_{\lambda}<2d}m_{\lambda}\cdot\mathrm{dim}(V_{\lambda}),
\end{equation*}
where $m_{\lambda}:=\mathrm{dim}\,\mathrm{Hom}_K(V_{\lambda},W^{\mathbb{C}})$ and $\pi:K\rightarrow\mathrm{Aut}(W)$ is again the slice representation of $\K \P^2=G/K\subset\mathbb{S}^n$. An  analogous formula with 
$c_\lambda = 2d$, holds for the nullity. Thus, we have to determine all representations $\lambda\in D(G)$ for all $\K \P^2$ with Casimir eigenvalue less or equal to $2d$. From Proposition \ref{Laplace is Casimir} we know that the Casimir operator 
coincides with the standard Laplace operator on the three normal bundles. Moreover, the standard Laplacian commutes with parallel
bundle maps and only depends on the representation defining the bundle. Hence, we can use spectral computations for the
Hodge-Laplace operator from the literature, if the slice representations (\ref{slice representations}) appear.

\medskip

\noindent
\textbf{Case $\C \P^2$:} In \cite{Ikeda1978}, the eigenvalues of the Laplacian of $\C \P^n$ on $\Lambda_0^{p,q}$ have been calculated.
In our case, we have $p=q=1$ and $n=2$. The calculations on p.532-533 give the following possible representations and corresponding eigenvalues of $\mathrm{Cas}_{\lambda}^{\mathrm{SU}(3),b_{\mathbb{C}}}$:

\bgroup
\def\arraystretch{1.5}
\begin{table}[H]
\!\!\!\!\!\!\!\!\begin{tabular}{|c|c|}
\hline
$\mathrm{SU}(3)$-representation $\lambda$ & Casimir eigenvalue $\mathrm{Cas}_{\lambda}^{\mathrm{SU}(3),b_{\mathbb{C}}}$ \\ \hline \hline
$(k+1)\omega_1+(k+1)\omega_2$\quad($k\geq0$) & $\frac{4}{3}(k+1)(k+3)$ \\ \hline
$(k-1)\omega_1+(k+2)\omega_2$\quad($k\geq1$) & $\frac{4}{3}(k+1)(k+2)$ \\ \hline
$(k+3)\omega_1+k\omega_2$\quad($k\geq0$) & $\frac{4}{3}(k+2)(k+3)$ \\ \hline
\end{tabular}
\end{table}
\egroup

All these representations $\lambda$ have multiplicity $m_{\lambda}=1$. Note that the eigenvalues in \cite{Ikeda1978} 
are calculated with respect to the Fubini-Study metric, i.e. the one that is induced by $\frac{1}{3}b_{\su(3)}$. Since
$b_\C = \frac14 b_{\su(3)}$ we have to scale the numbers in \cite{Ikeda1978}   by the factor $\frac43$. Thus, the representations that contribute to the nullity are precisely $3\omega_2$ ($k=1$ in the second family of representations in the above table) and $3\omega_1$ ($k=0$ in the third family). The single representation that contributes to the index is $\omega_1+\omega_2$ ($k=0$ in the first family). The result follows from $\mathrm{dim}(V_{3\omega_2})=10$, $\mathrm{dim}(V_{3\omega_1})=10$, $\mathrm{dim}(V_{\omega_1+\omega_2})=8$ and 
\begin{equation*}
\mathrm{Nul}_{\mathrm{K}}(\C \P^2)=\mathrm{dim}(\mathrm{SO}(8))-\mathrm{dim}(\mathrm{SU}(3))=28-8=20.
\end{equation*}

\medskip

\noindent
\textbf{Case $\H \P^2$:} In \cite{Tsukamoto1981}, the eigenvalues of the Laplacian of $\H \P^n$ on $\Lambda_0^p$ have been calculated.
In our case, we have $p=2$ and $n=2$. The table on p.\ 425 gives the following possible representations and corresponding eigenvalues of $\mathrm{Cas}_{\lambda}^{\mathrm{Sp}(3),b_{\mathbb{H}}}$:

\bgroup
\def\arraystretch{1.5}
\begin{table}[H]
\!\!\!\!\!\!\!\!\begin{tabular}{|c|c|}
\hline
$\mathrm{Sp}(3)$-representation $\lambda$ & Casimir eigenvalue $\mathrm{Cas}_{\lambda}^{\mathrm{Sp}(3),b_{\mathbb{H}}}$ \\ \hline \hline
$k\omega_2$\quad($k\geq1$) & $\frac{4}{3}k(k+5)$ \\ \hline
$\omega_1+k\omega_2+\omega_3$\quad($k\geq0$) & $\frac{4}{3}(k^2+8k+12)$ \\ \hline
\end{tabular}
\end{table}
\egroup
All these representations $\lambda$ have multiplicity $m_{\lambda}=1$. The eigenvalues in \cite{Tsukamoto1981} 
are computed with respect to the metric induced by $b_{\sp(3)}$. But since $b_\H = \frac{3}{32}b_{\sp(3)}$ we have to scale
the numbers in \cite{Tsukamoto1981} by the factor $\frac{32}{3}$ to get the right result.
%
There is a small computational mistake in \cite{Tsukamoto1981} regarding the eigenvalue of the $(-16)\cdot$Casimir operator on the irreducible $\mathrm{Sp}(3)$-module $I(r,s,t)$: it should be $12t$ instead of $6t$. Thus, the representation that contributes to the index is $\omega_2$ ($k=1$ in the first family), and the representation that contributes to the nullity is $\omega_1+\omega_3$ ($k=0$ in the second family). The result follows from $\mathrm{dim}(V_{\omega_2})=14$, $\mathrm{dim}(V_{\omega_1+\omega_3})=70$ and the fact that
\begin{equation*}
\mathrm{Nul}_{\mathrm{K}}(\H \P^2)=\mathrm{dim}(\mathrm{SO}(14))-\mathrm{dim}(\mathrm{Sp}(3))=91-21=70.
\end{equation*}

\medskip

\noindent
\textbf{Case $\O \P^2$:} In \cite{Mashimo1997}, the eigenvalues of the Laplacian of $\O \P^2$ on $\Lambda^p$ have been calculated.
In our case we have $p=1$. Theorem 6 and the table on p.\ 393 give the following possible representations and corresponding eigenvalues of $\mathrm{Cas}_{\lambda}^{\mathrm{F}_4,b_{\mathbb{O}}}$:

\medskip


\bgroup
\def\arraystretch{1.5}
\begin{table}[H]
\!\!\!\!\!\!\!\!\begin{tabular}{|c|c|}
\hline
$\mathrm{F}_4$-representation $\lambda$ & Casimir eigenvalue $\mathrm{Cas}_{\lambda}^{\mathrm{F}_4,b_{\mathbb{O}}}$ \\ \hline \hline
$k\omega_4$\quad($k\geq1$) & $\frac{4}{3}(k^2+11k)$ \\ \hline
$\omega_3+k\omega_4$\quad($k\geq0$) & $\frac{4}{3}(k^2+14k+24)$ \\ \hline
\end{tabular}
\end{table}
\egroup
All these representations $\lambda$ have multiplicity $m_{\lambda}=1$. Again, we have to be careful, because the eigenvalues in \cite{Mashimo1997} 
are computed with respect to the metric induced by $\frac{1}{18} b_{\f4}$. But since $b_\O = \frac{1}{24} b_{\f4}$ we have to
scale the numbers in \cite{Mashimo1997}  by the factor $\frac43$.
%
Thus, the representation that contributes to the index is $\omega_4$ ($k=1$ in the first family), and the representation that contributes to the nullity is $\omega_3$ ($k=0$ in the second family). The result follows from $\mathrm{dim}(V_{\omega_4})=26$, $\mathrm{dim}(V_{\omega_3})=273$ and the fact that
\begin{equation*}
\mathrm{Nul}_{\mathrm{K}}(\O \P^2)=\mathrm{dim}(\mathrm{SO}(26))-\mathrm{dim}(\mathrm{F}_4)=325-52=273.
\end{equation*}

In each of the three cases we see that there is always exactly one  $G$-repr\-esentation, which contributes 
to the index and this representation has Casimir eigenvalue $d$ and dimension $\frac32d + 2= n+1$. This proves 
the first statement  of Theorem \ref{Main Theorem}. The second statement follows by collecting the three computations of 
the nullities.
\end{proof}

\bibliographystyle{plain}

\end{document}